\documentclass[12pt, twoside, a4paper]{article}
\usepackage{etoolbox}
\usepackage[top=1.5in, bottom=1in, left=1in, right=1in]{geometry}
\usepackage{amsfonts, amsmath, amssymb, amsthm, amscd}
\usepackage{abstract}
\usepackage{lmodern}
\usepackage{fancyhdr}
\usepackage{mathtools}
\usepackage{mathrsfs} 
\usepackage{pst-node}
\usepackage{bigints}
\usepackage{mathrsfs}
\usepackage{titlesec}
\usepackage[normalem]{ulem}

\usepackage{tikz-cd}

\usetikzlibrary{arrows}
\newtheorem{theorem}{Theorem}
\newtheorem*{theorem*}{Theorem}
\numberwithin{theorem}{subsection}
\numberwithin{theorem}{section}
\newtheorem{lemma}{Lemma}
\numberwithin{lemma}{section}
\newtheorem{proposition}{Proposition}
\numberwithin{proposition}{section}

\numberwithin{corollary}{section}
\numberwithin{equation}{section}
\numberwithin{equation}{subsection}

\allowdisplaybreaks

\newcommand{\R}{\mathbb{R}}

\newcommand{\J}{\mathcal{J}}

\titleformat*{\section}{\large}
\titleformat{\subsection}[runin]
  {\normalfont\normalsize}{\thesubsection}{1em}{}

\titleformat{\subsubsection}[runin]
  {\normalfont\small}{\thesubsubsection}{1em}{}
\titleformat*{\paragraph}{\large\bfseries}
\titleformat*{\subparagraph}{\large\bfseries}
\newtheorem*{theorem-non}{Theorem}

\title{\normalsize \textbf{NON-ISOMETRIC PAIRS OF RIEMANNIAN MANIFOLDS WITH THE SAME GUILLEMIN-RUELLE ZETA FUNCTION}}
\author{Hy P.G. Lam \footnote{Research partially supported by NSF RTG grant DMS-1502632.}}
\date{}

\fancyheadoffset[R]{0cm}
\begin{document}

\maketitle

\begin{abstract}
    {\sc Abstract. } In \cite{Su85}, T. Sunada constructed a vast collection of non-isometric Laplace-isospectral pairs $(M_1,g_1)$, resp. $(M_2,g_2)$ of Riemannian manifolds. He further proves that the Ruelle zeta functions $Z_g(s):= \prod_{\gamma}(1 - e^{-sL(\gamma)})^{-1}$ of $(M_1,g_1)$, resp. $(M_2,g_2)$ coincide, where $\{\gamma\}$ runs over the primitive closed geodesics of $(M,g)$ and $L(\gamma)$ is the length of $\gamma$. In this article, we use the method of intertwining operators on the unit cosphere bundle to prove that the same Sunada pairs have identical Guillemin-Ruelle dynamical L-functions 
    $\textup{L}_G(s) = \sum_{\gamma\in \mathscr{G}}\frac{L_\gamma^\# e^{-sL_\gamma}}{|\textup{det}(I -\mathbf{P}_\gamma)|}$, where the sum runs over all closed geodesics. 
\end{abstract}

\section{\centering \sc{Introduction}} 
~~~The purpose of this article is to use a modified Sunada construction to produce non-isometric pairs with the same Guillemin-Ruelle dynamical zeta function. To state the result, we need to introduce some background and notation. Assume that a connected compact manifold $M_0$ fits as the base of a diagram of finite normal covers,
\begin{equation}\label{maindiag}
  \begin{tikzcd}
& (M,\pi^*g_0) \arrow[dl,swap,"\pi_1"] \arrow[dd,"\pi"] \arrow[dr,"\pi_2"]\\
(M_1:= H_1\backslash M,p_1^*g_0)  \arrow[dr,swap,"p_1"] &&
  (M_2: = H_2\backslash M,p_2^*g_0) \arrow[dl,"p_2"] \\
& (M_0 = G\backslash M,g_0)
\end{tikzcd}  
\end{equation}
where the triplet $(G, H_1 ,H_2) $ satisfies $L^2 (G/H_1) \simeq L^2(G/H_2)$ as unitarily equivalent $G$-modules. Suppose that $G$ acts freely on $M$. Sunada then proves that for any metric $g_0$ on $M_0$, the induced metrics $g_i := p_i^*g_0 ~(i= 1,2)$ are Laplace-isospectral \cite{Su85}. In addition, the pairs have the same Ruelle length L-function $\textup{L}_g(s) = \sum_{\gamma\in\mathscr{G}} e^{-sL_\gamma}$, where the sum runs over all closed geodesics (see \cite{Su94} for more geometric descriptions of $L$-functions).

\medskip

Diagram \eqref{maindiag} induces a diagram of finite covers of cosphere bundles,
\begin{equation}\label{2diag}
  \begin{tikzcd}
& S^*M \arrow[dl,swap,"\widetilde{\pi}_1"] \arrow[dd,"\widetilde{\pi}"] \arrow[dr,"\widetilde{\pi}_2"]\\
S^*M_1 \arrow[dr,swap,"\widetilde{p}_1"] &&
  S^*M_2. \arrow[dl,"\widetilde{p}_2"] \\
& S^*M_0 
\end{tikzcd} 
\end{equation}

For each cosphere bundle of \eqref{2diag}, we denote the contact Hamiltonian flow by $G^t_{g}$ (\S \ref{Hamilflow}), and the set of its periodic orbits $\gamma = \gamma(t)$ by $\mathscr{G}$. We denote $L_\gamma$ and  $ L_\gamma^\#$ respectively the length and prime period of $\gamma$. Let $\mathbf{P}_{\gamma}$ be the associated linearized Poincar\'{e} map (\S \ref{Poincaremap}).  $G^t_g$ is said to be \textit{Lefschetz} if all of its periodic orbits are non-degenerate, that is, the maps $I-\mathbf{P}_{\gamma}$ are invertible. Under this hypothesis, we define the Guillemin-Ruelle zeta function with respect to a given Riemannian metric $g$ by 

\begin{equation}\label{Lfunction}
    \textup{L}_G(s) =  \sum_{\gamma\in\mathscr{G}} \frac{L^\#_\gamma e^{-sL_\gamma}}{|\det(I
    - \mathbf{P}_{\gamma})|}
\end{equation}

More generally, a geodesic flow is said to be \textit{clean} if its fixed point sets are Bott-Morse non-degenerate (\S \ref{cleanfix}). For a fixed $t>0$, we let $\text{Fix}(G^t_g) := \{\zeta \in S^*M: G_g^t \zeta = \zeta\} $ be the fixed points set of $G^t_g$. Under the clean fixed point condition, $\textup{Fix}(G^t_g) = Z_1 \cup \ldots \cup Z_d$, each $Z_j$ is a submanifold of $S_g^*M$ and $T_\zeta Z_j = \textup{Fix}((dG_g^t)_\zeta)$. The Poincar\'{e} map is replaced by the \textit{normal} Poincar\'{e} map $ P_{\gamma}^{\#}$ 
in (\ref{Lfunction}) (see \S \ref{canvolZj}), which is simply the bijective mapping induced by $dG^t_g$
\begin{equation}\label{normalPoincare}
(I - P_{Z_j}^{\#}): = (I - dG^t_g)^\#: T S^*M  / T Z_j \rightarrow T S^*M / T Z_j.
\end{equation}

Furthermore, on each $Z_j$, there is a Gelfand-Leray form
\[
\int_{Z_j}d\nu_{Z_j}(\zeta) = \int_{Z_j}\frac{d\mu_L}{|\det(I - \mathbf{P}_{Z_j}^\#)| \hat{\beta}}
\]
where $\mu_L$ is the normalized Liouville measure on $S_g^*M$ and $\hat{\beta}$ is the volume density on the normal space $NZ_j$ determined by the defining equation $dG^t_g \zeta = \zeta $ (\S \ref{deltadensity}, see also \S \ref{stationaryphase}).


Let $\mathscr{F}$ be the union of all component of all fixed point sets. Let $L_Z$ be the common length of all closed geodesics in $Z$ so that $G^{L_Z}_g$ fixes $Z$ and denote $d\mu_Z := \hat{\beta} \lrcorner d\mu_L$.

\begin{equation}\label{Lfunction2}
    \textup{L}_G(s) =  \sum_{Z\in\mathscr{F}} e^{-sL_Z}\int_Z |\det(I
    - \mathbf{P}_Z^\#)|^{-1}d\mu_Z
\end{equation} 

\begin{theorem} \label{Maintheorem} Let $\pi: (M,\pi^*g_0) \rightarrow (M_0,g_0)$ be as in the Sunada diagram \eqref{maindiag}. Let $G^t_{g_i}$ be (contact) Hamiltonian flows on each cosphere bundle $(S^*M_i,g_i)$, $i=1,2$. Let $L_{G_i}$ be the Guillemin-Ruelle  L-function associated to $g_i$, $i=1,2$. If the flows $G^t_{g_i}$ are clean and the triplet $(G, H_1 ,H_2) $ satisfies $L^2 (G/H_1) \simeq L^2(G/H_2)$ as unitarily equivalent $G$-modules $(\dagger)$, then $\textup{L}_{G_1} = \textup{L}_{G_2}$.
\end{theorem}

To prove Theorem \ref{Maintheorem}, we first show that the Koopman operators associated to the (co)-geodesic flows $G^t_{g_i}$ on the pair $S^*M_i$, $i=1,2$ from \eqref{maindiag} are unitarily equivalent by a unitary Fourier integral operator. The Koopman operators on $L^2(S^*M_i, d\mu_{L_i})$ are defined by $V_{g_i}^t f = f \circ G^t_{g_i}$ (\S \ref{KoopmanOp}). The unitary Fourier integral operators are constructed via a generalization of the technique developed by S. Zelditch in \cite{Z92}. In the article, the author provides an explicit intertwining operator defined by a finite Radon transform associated to the diagram of covers in \eqref{maindiag}. By definition, a finite Radon transform is a Fourier integral operator associated to a finitely multi-valued symplectic correspondence. In our case, we define a finite Radon transform on the unit cosphere bundle $\widetilde{U}_A : L^2(S^*M_1,dx_1)\rightarrow L^2(S^*M_2,dx_2)$ associated to the induced covering diagram \eqref{2diag}.

Here, we identify the sphere bundle to the copshere bundle via the metric. $\widetilde{U}_A$ has the form
\begin{equation}\label{liftKoop}
\widetilde{U}_A= \frac{1}{\# H_1}\sum_{a\in G} A(a)\widetilde{\pi}_{2*}\widetilde{T}_{a}\widetilde{\pi}_{1}^*: L^2(S^*M_1, d\mu_{L_1}) \rightarrow L^2(S^*M_2,d\mu_{L_2})
\end{equation}
where $A$ is the function on $H_2\backslash G/H_1$ inducing the intertwining operator $\mathcal{A}:  L^2 (G/H_1) \rightarrow L^2(G/H_2)$ (\S\ref{UIO}). The translations $\{T_a\}_{a\in G}$ are local isometries that lift to a map $\widetilde{T}_a: S^*M \rightarrow S^*M$ where $\widetilde{T}_a(x,\xi) = \left(a\cdot x, (a^{-1})^* \cdot \xi\right) $.

\begin{theorem}\label{Vtheorem}
Let $\pi: (M,\pi^*g_0) \rightarrow (M_0,g_0)$ be a normal finite Riemannian covering with covering transformation group G as in \eqref{maindiag}.  If the triplet $(G, H_1, H_2)$ satisfies $(\dagger)$, then the Koopman operators $V^t_{g_i} : C^\infty_0 S^*M_i \rightarrow C^\infty_0 S^*M_i$ associated with the subcovers $(M_i, g_i)$ for each $i=1,2$, are unitarily equivalent via $\widetilde{U}_A$. Namely,  \[\widetilde{U}_AV^t_{g_1}\widetilde{U}^*_A = V^t_{g_2}\]
\end{theorem}

The Schwartz kernel $k^t(\xi,\eta)$ of the Koopman operator $V^t_g$ is a $\delta$-function $\delta(\xi - G^t_g \eta) $. We refer to \cite{G77} for the use of the Schwartz kernel theorem in this statement.  Guillemin further defined the flat trace of $V^t_g$ formally by 
\begin{equation}\label{flatdef}
    \textup{Tr}^\flat V^t_g = \int_{S^*M}\delta(\xi - G^t_g \xi) d\mu_L(\xi) 
\end{equation}
More precisely, the trace is defined as a pushforward-pullback operation on distributions by  $\Pi_*\Delta^*\delta(\xi - G^t_g\xi)$ where $\Delta: S^*M\rightarrow S^*M\times S^*M$ is the diagonal embedding, and $\Pi: S^*M \times \R^+ \rightarrow S^*M$ is the natural projection. 

\medskip 

If all periodic trajectories of the geodesics flows $G^t_{g_i}$ are non-degenerate, then the righthand side of \eqref{flatdef} becomes  \cite{G77}
\begin{equation}\label{Guilleminflat}
\textup{Tr}^\flat V^t_g = \sum_{\gamma \in \mathscr{G}}\frac{L_{\gamma}^\#\delta(t-L_\gamma)}{|\textup{det}(I - \textbf{P}_\gamma)|}, ~ t > 0 
\end{equation}
In \S\ref{proofcleanprop}, we show that \eqref{Guilleminflat} has a natural extension to geodesic flows with clean fixed point sets.

\begin{proposition}\label{cleanprop}Assume all of the fixed points of $G^t_g$ are clean, then
\begin{equation}\label{sumflattrace}
\textup{Tr}^\flat V^t_g = \sum_{Z\in\mathscr{F}}\delta(t - L_Z)\int_Z \frac{d\mu_Z}{|\det(I- \mathbf{P}_Z^\#)|}
\end{equation}
\end{proposition}

\medskip

We show in \S \ref{proofofsameflattrace} that the Riemannian pairs resulting from Sunada's construction have the same Guillemin flat trace, namely, $\textup{Tr}^\flat V^t_{g_1} = \textup{Tr}^\flat V^t_{g_2}$. Theorem \eqref{Maintheorem} follows from this result as a direct corollary. Here, it is important to note that the flat trace is not an $L^2$-trace and therefore it does not follow from Theorem \ref{Vtheorem} that $V^t_{g_1}, V^t_{g_2}$ have the same flat trace.


\textbf{Remark.}  Proposition \ref{cleanprop} is the analogue for the Koopman operator of the Duistermaat-Guillemin wave trace formula. In the case of Lefschetz flows, the formula states \cite{GuDu}
\begin{equation}\label{wavetrace}
    \textup{Tr} \left( \exp\left(it\sqrt{\Delta_g}\right)\right) = \sum_{j=1}^\infty e^{i\lambda_jt} = \sum_{\gamma \in \mathscr{G}}\frac{i^{\sigma_\gamma}L_\gamma^\#\delta(t- L_\gamma)}{|\textup{det}(I - \textbf{P}_{\gamma})|^{\frac{1}{2}}} + L^1_{loc}, ~~ t \in \R
\end{equation}
with respect to the Laplace-Beltrami operator $\Delta_g: C^\infty M \rightarrow C^\infty M$, where $\lambda_j$ are the eigenvalues of $\sqrt{\Delta_g}$, and $\sigma_{\gamma}$ is the Maslov index of $\gamma$. Laplace isospectrality is equivalent to
\begin{equation}\label{wavetraceiso}
    \textup{Tr} \left( \exp\left(it\sqrt{\Delta_{g_1}}\right)\right) = \textup{Tr} \left( \exp\left(it\sqrt{\Delta_{g_2}}\right)\right),
\end{equation}
and by \eqref{wavetrace}, 
\begin{equation}\label{wavetraceiso2}
    \sum_{\gamma \in \mathscr{G}_ {g_1}}\frac{i^{\sigma_\gamma}L_\gamma^\#\delta(t- L_\gamma)}{|\textup{det}(I - \textbf{P}_{\gamma})|^{\frac{1}{2}}} = \sum_{\gamma \in \mathscr{G}_{g_2}}\frac{i^{\sigma_\gamma}L_\gamma^\#\delta(t- L_\gamma)}{|\textup{det}(I - \textbf{P}_{\gamma})|^{\frac{1}{2}}}
\end{equation}

S.Chen and A.Manning \cite{ChMa} studied the associated L-function
\begin{equation*}
    \log Z_g(s) = \sum_{\gamma \in \mathscr{G}}\sum_{k=1}^\infty \frac{e^{-skL^\#_\gamma}}{k|\textup{det}(I - \textbf{P}^k_\gamma)|^{\frac{1}{2}}} 
\end{equation*}
where $\textbf{P}_{\gamma}^k $ being the linearized Poincar\'{e} map corresponding to the $k$-fold iterate of $\gamma$. It follows from \eqref{wavetraceiso} and \eqref{wavetraceiso2} that isospectrality implies 
$\log Z_{g_1} = \log Z_{g_2}$. However, if there are multiplicities in the length spectrum of $S^*M$, one cannot determine the individual term of each length from the sum, and therefore one cannot determine $\textup{L}_G$ from \eqref{wavetraceiso}, \eqref{wavetraceiso2}.

\subsection{\normalsize \textbf{Convergence of the L-function. }} For compact negatively curved manifolds, Chen-Manning  \cite{ChMa} proved that the zeta function $Z_g(s)$ converges absolutely for $\textup{Re }s > P(\frac{1}{2}\alpha)$, where $\alpha$ is the volume growth rate in the orientable unstable foliation on the cotangent bundle of $S^*M$, and $P$ is the pressure of Sinai's function on $S^*M$. It follows from their theorem that $\textup{L}_G(s)$ converges absolutely on the right half-plane $\{\textup{Re }s > P(\alpha)\}$ under the same assumption. It is proven by Dyatlov-Zworski that $Z_g(s)$ has a meromorphic continuation to the whole complex plane \cite{DyZw} (see also \cite{GLP}). There are many articles concerned with the meromorphic continuation of the Ruelle zeta function; we refer to \cite{DyZw} for references and background.

\medskip

A natural problem is to show that there is a half-plane of convergence for $\textup{L}_G(s)$ for any compact Riemannian manifold and to study its memormorphic continuation. In the forthcoming article \cite{H22}, we will use a heat kernel approach to flat traces to prove that $\textup{L}_G(s)$ converges in $\{Re(s) > 0\}$ at least for the cases of flat tori, spheres, compact hyperbolic manifolds and for Heisenberg nilmanifolds. In work in progress, we extend the result to general non-positively curved manifolds.

\subsection{\normalsize \textbf{Acknowledgment. }}
This article is part of the author's PhD thesis at Northwestern University under the guidance of Steve Zelditch.

\section{\centering \sc{Preliminaries}}

\subsection{\normalsize  \textbf{Contact Hamiltonian flow}.}\label{Hamilflow} Consider $G^t_g: (S_g^*M,\mu_L) \rightarrow (S_g^*M, \mu_L)$ to be the geodesic flow on the unit cosphere bundle generated by the quadratic form 
\[
H(x,\xi) = \frac{1}{2}|\xi_x|^2_{g}=\frac{1}{2}g^{ab}(x)\xi_a\xi_b
\]
with $g^{ab}(x)$ is the inverse metric tensor via $g^{ab}(x)g_{bc}(x) = \delta^a_c$, and $\xi_x$ is a covector in $S^*_xM$ realized under the local trivialization of $S^*M$ restricted to a coordinate neighborhood $U$, in which a 1-form $\zeta = \xi_adx^a$ in  $S_x^*M|_U$ is identified with the point $(x,\xi_a) \in U \times \mathbb{S}^{n-1}$.

In local (Darboux) coodinates, $G^t_g$ satisfies the geodesic equations
\[
\dot{x}^a = \frac{\partial H}{\partial \xi_a} = g^{ab}(x)\xi_b~; ~~
\dot{\xi}_a = -\frac{\partial H}{\partial x^a} = -\frac{1}{2}\frac{\partial g^{bc}(x)}{\partial x^a}\xi_a\xi_b.
\]
and is a contact Hamiltonian flow on $S^*_gM$, which projects to a geodesic line on $M$. 

The Liouville measure $\mu_L$ on $S^*M$ is the Leray measure $\frac{(dx \wedge d\xi)^n}{dH}$ where $(dx\wedge d\xi)^n$ is the symplectic volume measure. The canonical 1-form $\alpha$ on $T^*M$ is given by $\alpha_{(x,\xi)}(v) = \xi_x(d\pi(v))$ for $\pi: T^*M \rightarrow M$.

Let $\eta = \alpha\big|_{S^*M}$ be the contact form obtained by restricting $\alpha$ to the cosphere bundle. The Reeb vector field $R_\eta$ is uniquely determined by 
\[
\eta(R_\eta) = 1 ~~\textup{and}~~ R_\eta \lrcorner d\eta = 0.
\]
The contact Hamiltonian vector field $X_H$ is then defined by 
\[
\eta(X_H) = - H ~~\textup{and}~~ X_H \lrcorner d\eta = dH - \mathcal{L}_{X_H}\eta.
\] 
where $\mathcal{L}_{X_H}\eta$ is the Lie derivative of $\eta$ along the vector field $X_H$.

\subsection{\normalsize  \textbf{Poincar\'{e} map}.}\label{Poincaremap}
 Let $\gamma = \gamma(t)$ be a periodic trajectory of $G^t_g$ with period $\tau$ on $S^*M$
starting at $\zeta$. Let $\Sigma $ be a transversal local hypersurface to the flow at $\zeta$. Let $T$ be the first return time of an orbit starting from $\Sigma$ to $\Sigma$. The first-return map is the following 
\[\Phi_{\Sigma}: \Sigma \rightarrow \Sigma; ~\eta \mapsto  G_g^{T(\eta)} \eta.\]

At the periodic point $\zeta$, the linear Poincare map is defined  by
\begin{equation}\label{Poincare}
    P_{\gamma} = (d\Phi_\Sigma)_\zeta: T_\zeta \Sigma \rightarrow T_\zeta \Sigma
\end{equation}
It is useful to give a description of the Poincare map in terms of Jacobi fields (cf. \cite{Su94}),

\begin{equation*}
\mathbf{P}_{\gamma} = \begin{pmatrix}
\textbf{A}& \textbf{B}\\
&\\
\textbf{C} & \textbf{D}
\end{pmatrix} = \begin{pmatrix}
(\langle v_i ,U_j(\tau) \rangle)_{i,j} & (\langle v_i, V_j(\tau)\rangle)_{i,j} \\
&\\
(\langle v_i, \nabla_{\dot{\gamma}} U_j(\tau) \rangle)_{i,j} & (\langle v_i, \nabla_{\dot{\gamma}} V_j(\tau) \rangle )_{i,j}
\end{pmatrix} 
\end{equation*}
where $U_i(t), V_i(t)$ are the basis of normal Jacobi fields along $\gamma$ with $U_i(0) = v_i$ , $\nabla U_i(0) = 0$, $V_i(0) = 0 $, $\nabla V_i(0) = v_i $ where  $\{v_i\}$ is an orthonormal basis of $T_\zeta \Sigma$ (see \S 3.4). By Schur's formula, 
\begin{equation*}
    \det (I - \mathbf{P}_{\gamma})= \det( I-\textbf{D} )  \det ((I -\textbf{A})-  \textbf{B}(I-\textbf{D})^{-1}\textbf{C}) 
\end{equation*}

A periodic geodesic orbit $\gamma$ is \textit{non-degenerate} if $\det (I -\mathbf{P}_{\gamma})\neq 0$. This means that the map $I -\mathbf{P}_\gamma$ does not have 1 as its eigenvalue. Examples of non-degeneracy are when the eigenvalues are of the following three types: (i) $e^{\pm i \alpha_j}$ where the $\alpha_j$'s and $\pi$ are independent over $\mathbb{Q}$ (elliptic); (ii) $e^{\pm \lambda_j}$ where $\lambda_j \in \R\backslash \{0\}$ (hyperbolic); (iii) $e^{\pm\mu_j\pm i \nu_j}$ with $\mu_j,\nu_j\in \R$ not all zero (loxodromic). More generally, we are interested in clean fixed point sets.



\subsection{\normalsize  \textbf{ Clean fixed point set of $G^t_g$.}}\label{cleanfix} We denote the fixed point set of $dG_g^t : TS_g^*M \rightarrow TS_g^*M$ by $\textup{Fix}(dG_g^t)$ note that $T\text{Fix}(G^t_g) \subset\textup{Fix}(dG_g^t)$ as $G_g^t$ is the identity map on $\text{Fix}(G^t_g)$. In general, the fixed point set of $dG^t_g$ can be larger than $T\textup{Fix}(G^t_g)$, or equivalently, there can exist normal Jacobi fields that do not come from varying the periodic geodesics in the fixed point set.   

As defined above, $G_g^t$ is said to have a \textit{clean} fixed point set for period $\tau$ if 
 $T  \text{Fix}(G^\tau_g) = \textup{Fix}(dG^\tau_g)$.

Assume $Z$ is a clean component of period $\tau$ where  all points in $Z$ share the same periodic trajectory, then a local transversal region to the geodesic flow restricted to a neighborhood about a point $\zeta \in Z$ is a hyperspace $\mathbf{\Sigma}$ with codimension strictly greater than 1.

The manifold $M$ is said to satisfy the \textit{Clean Intersection Hypothesis} if $G^t_g$ has a clean fixed point set for all period $\tau \in
\textup{Lsp}(S_g^*M)$.



\subsection{\normalsize \textbf{Canonical volume density on a clean submanifold and the normal \linebreak Poincar\'{e} map. }}\label{canvolZj}Suppose $Z$ is a clean component of $\textup{Fix}(G^\tau_g)$ consisting of periodic orbits of period $\tau \in \textup{Lsp}(S_g^*M)$.
A periodic orbit in $Z$  projects under $\pi$ to a closed geodesic $\bar{\gamma}(t)$ on $M$.

The geodesic variations coming from the family of closed geodesics on $Z$ form the total space of Jacobi fields on $\pi(Z)$. Given a Jacobi field $\bar{J}$ in this space, $\bar{J}$ lifts to a flow-invariant field on $S^*M$ along $\gamma$ via the correspondence
\[
\bar{J}\rightarrow (\bar{J}^h , \nabla_{\dot{\gamma}}\bar{J}^v ) 
\]
where $\bar{J}^h$, $\nabla_{\dot{\gamma}}\bar{J}^v$ are respectively the horizontal and vertical lifts of $\bar{J}$ and $\nabla_{\dot{\gamma}}\bar{J}$ [Lemma 3.1.6, \cite{Kl}]. 

\smallskip

Denote the total space of periodic Jacobi fields on $Z$ along $\gamma$ by $\mathcal{J}_Z(\gamma)$.  Given $X, Y \in \mathcal{J}_Z(\gamma)$, 
the Riemannian inner product is 
\begin{equation}\label{metricJfields}
    \langle X , Y \rangle = \int_{\bar{\gamma}} g\left(d\pi (X_h(s)), d\pi (Y_h (s))\right)ds 
\end{equation}
where $X_h$, $Y_h$ are the horizontal components of $X$ and $Y$ via the splitting induced from the Riemannian connection on $S_g^*M$. 

Denote further $\mathcal{J}^\parallel_Z(\gamma)$ to be the subspace of periodic J-fields that are $Z$-tangential along $\gamma$. This subspace spans $T\big|_{\gamma} Z$. Also, we have an orthonormal basis of $\mathcal{J}^\parallel_Z(\gamma)$ including $\dot{\gamma}$, and hence obtain a collection of 1-forms $\beta^1, \ldots, \beta^k$ on $Z$ corresponding to the dual basis, which yields
\begin{equation}\label{volcanmeasure}
 |d\textup{Vol}_{can}(Z)|(\zeta) = \left|\beta^1_\zeta \wedge \ldots \wedge \beta^k_\zeta \right|
\end{equation}
as a natural volume element on $Z$.
We should like to point out that there is a symplectic structure on the space of all Jacobi fields along $\gamma$ given by the Wronskian 
\begin{equation*}
    \omega(J, J') = g(J, \nabla_{\dot{\gamma}} J') - g( \nabla_{\dot{\gamma}} J, J').
\end{equation*}

Now denote $\mathcal{J}_Z^\perp (\gamma) := \mathcal{J}_Z(\gamma)/\mathcal{J}_Z^\parallel(\gamma)$ to be the space of periodic Jacobi fields along $\gamma$ that are transversal to $Z$. This space spans $T\big|_{\gamma}S^*_gM/ T\big|_{\gamma}Z$. The normal Poincare map is
 realized as the linear symplectic transformation on $\mathcal{J}_Z^\perp (\gamma)$, which is given by 
\[
(\textbf{P}^\#_Z)_\zeta: T_\zeta S^*_gM/ T_\zeta Z \oplus T_\zeta S^*_gM/ T_\zeta Z \circlearrowleft_{\textup{translate by~} t =\tau}\]
\[
(J(0) , \nabla_{\dot\gamma}J(0)) \mapsto (J(\tau), \nabla_{\dot\gamma} J(\tau))\]
where $J(t) \in \J_Z^\perp (\gamma)$.  Let $\nu_1, \ldots, \nu_{2n-1 - k}$ $(k = k(j,\tau) = \dim Z )$ be an orthonormal basis of $T_{\zeta}S^*_gM/ T_{\zeta}Z$ along $\gamma$ with respect to the metric \eqref{metricJfields}. Let $U_i(t), V_i(t)$ be an orthonormal basis of $\J^\perp_Z(\gamma)$ with the initial conditions 
\begin{equation}\label{Jbasis}
     \begin{cases}
          U_i(0) = \nu_i ~, ~ V_i(0) = 0\\
          \\
           \nabla_{\dot{\gamma}}U_i(0) = 0 ~,~ \nabla_{\dot{\gamma}} V_i(0) = \nu_i. \\
    \end{cases}       
\end{equation}
In block-matrix form, the transformation of $U_i, V_i$ by $(P^\#_Z)_\zeta$ is 
\begin{equation}\label{normalPoincareinJfields}
 (\textbf{P}^\#_Z)_\zeta=  \begin{pmatrix}
    (\langle \nu_i, U_l(\tau)\rangle)_{i,l} & (\langle \nu_i, V_l(\tau) \rangle)_{i,l}\\
    & \\
    (\langle \nu_i, \nabla_{\dot\gamma}U_l(\tau)\rangle)_{i,l} & (\langle \nu_i, \nabla_{\dot\gamma}V_l(\tau) \rangle)_{i,l}
    \end{pmatrix}.
\end{equation}


\subsection{\normalsize \textbf{Stationary phase with Bott-Morse nondegenerate critical submanifold}.}\label{stationaryphase} We will need the clean version of the stationary phase formula and recall its statement in this section. For further background and proofs of the formula, we refer to \cite{GuSc13}, \cite{GuSc90}, \cite{Hor}. 

Let $(\mathcal{X},\mu)$ be an $N$-dimensional manifold equipped with a positive density measure. Let $a, \phi \in C^\infty ( \mathcal{X} ; \R)$ and denote 
\[
C_{\phi}: = \{ x \in \mathcal{X}: (X\phi)(x)  = 0, \forall X \in \Gamma(T\mathcal{X})\}
\]
to be the set of critical points of $\phi$. The function $\phi$ is said to be \textit{Bott-Morse} if $C_\phi$ is a submanifold of $\mathcal{X}$ and the transverse Hessian of $\phi$ denoted by $\textup{Hess}^\perp \phi$, is non-degenerate. 

\medskip
Consider the integral 
\begin{equation*}
    I(h) = \int_{\mathcal{X}} e^{ih\phi}a d\mu
\end{equation*}
with large parameter $h$. Let $\mathcal{Z}$ be a $k(\mathcal{Z})$-dimensional connected component of $C_\phi$, and let $y_1, \ldots, y_{N}$ be a system of coordinates on the normal bundle associated with $\phi$ defined by $y_1 = \ldots = y_{ k(\mathcal{Z})} = 0$.  The coordinates about a fixed point $x$ in $\mathcal{Z}$ are given by 
\[
y_1 = x_1 - \phi(x), \ldots, y_{k(\mathcal{Z})} = x_{k(\mathcal{Z})} - \phi(x), y_{k(\mathcal{Z})+1} = x_{ k(\mathcal{Z})+1}, \ldots, y_{N} = x_{N}.\]
With respect to the basis 
\[
\nu_{1} = \partial_{y_{k(\mathcal{Z})+1}}\big|_x,\ldots, \nu_{N - k(\mathcal{Z})} = \partial_{y_{N}}\big|_x 
\]
of the fibre $T_x\mathcal{X}/ T_x\mathcal{Z}$, the determinant of the transverse Hessian is 
\begin{align*}
    |\det \textup{Hess}^\perp\big|_{x} \phi| &= |d^2_p \phi (\nu_1 \wedge \ldots \wedge \nu_{N-k(\mathcal{Z})}, \nu_1 \wedge \ldots \wedge \nu_{N-k(\mathcal{Z})}  )| \\
    &= |\det (\partial^2_{y_i, y_j} \big|_{y=x}\phi )|.
\end{align*}
The density measure on $\mathcal{Z}$ is given by the quotient 
\[d\nu_{\mathcal{Z}} = \frac{d\mu}{|\det \textup{Hess}^\perp \phi|^{\frac{1}{2}}|dy_{k(\mathcal{Z})+1} \ldots   dy_N|} ~ .\]

\begin{lemma}\label{normalstatlemma}
If $\phi$ is a Bott-Morse function,  then 
\begin{equation}\label{normalstationary}
I(h) = \sum_{\mathcal{Z} \subset C_\phi }(2\pi h^{-1})^{\frac{N-k(\mathcal{Z})}{2}}
\left(e^{\frac{1}{4}i\pi \textup{sgn}(\mathcal{Z})}e^{ih\phi(\mathcal{Z})} \int_{\mathcal{Z}} ad\nu_{\mathcal{Z}} + O(h^{-1})\right),
\end{equation}
where $\textup{sgn}(\mathcal{Z})$ is the signature of the symmetric bi-linear form $D^2\phi \big|_{\mathcal{Z}}$. 

\end{lemma}


\subsection{\normalsize  \textbf{Koopman operator}.}\label{KoopmanOp} Let $G^t_g$ be a contact Hamiltonian geodesic flow on the unit copshere bundle $S^*M$ of a compact Riemannian manifold $(M,g)$. The \textit{Koopman operator}, denoted by $\{V^t_g\}_{t\geq0}$, is  defined as the one-parameter group of  unitary operators on $L^2(S^*M)$ given by
\begin{equation*}
    V^t_g: L^2(S^*M,d\mu_L)  \rightarrow L^2(S^*M,d\mu_L); ~ V^t_g( f) = f \circ G^t_g .
\end{equation*}

$V^t_g$ is unitary because $\mu_L$ is preserved by $G^t_g$.


\subsection{\normalsize  \textbf{Guillemin flat trace.}} The Schwartz kernel of the Koopman operator is given by 
\begin{equation*}
    V^\bullet_g f (\xi) = \int_{S^*M} f(\eta) \delta(\eta- G^\bullet_g  \xi ) d\mu_L(\eta),
\end{equation*}
that is, the kernel is a $\delta$-distribution in the sense of \cite{G77} with support on the graph of $G^t_g$. 
(see Appendix \S \ref{deltadensity}).
Recall the diagonal map 
\begin{equation*}
    \Delta :  S_g^*M \times \R^+ \rightarrow  S_g^*M\times S_g^*M \times \R^+
\end{equation*} 
$(\zeta,t) \mapsto ( \zeta, \zeta,t)$ and 
$\Pi: S_g^*M \times \R^+ \rightarrow S_g^*M$ the projection map $(\zeta ,t)\mapsto \zeta$,
Assuming the diagonal map intersects the graph of $G^\bullet_g$ transversally, Guillemin defines the flat trace 
\begin{equation} \label{second}
    \textup{Tr}^\flat V^\bullet_g = \Pi_*\Delta^* k^\bullet
\end{equation}
in Theorem 6, \cite{G77} and shows that it is a well-defined $\delta$-distribution on $\R^+$. The transversality hypothesis is equivalent to the Lefschetz condition imposed on the flow $G_g^t$ on $S_g^*M$ under which, for each $t>0$, the left-hand side of \eqref{second} is the formula \eqref{Guilleminflat}.
\medskip

Under the general condition of clean intersection (\S \ref{cleanfix}), \eqref{second} is still a well-defined $\delta$- distribution on $\R^+$ (\S \ref{deltadensity}). 


\subsection{\normalsize \textbf{Metric on the cosphere bundle.}}\label{KaluzaKleinmetric}
Instead of defining the normal space as a quotient as in \eqref{normalPoincare}, we can define it using as the normal bundle with respect to the Kaluza-Klein metric on $S^*M$. The Kaluza-Klein metric is defined by equipping a connection $\nabla$ on the cosphere bundle $S^*M$, which then determines a splitting of the tangent bundle $TS^*M$ into horizontal and vertical component

\[TS^*M = T^h S^*M \oplus T^v S^*M \]
where $T^v S^*M \cong S^{n-1}$ is given the standard  metric, and the natural projection  $d\pi: T^hS^*M  \rightarrow TM$ is an isometry, and $T^hS^*M$ is the kernel of the connection. That is, let $x$ be a point on $M$, for each fibre element $\zeta \in \pi^{-1}(x) \subset  S^*_xM$, the components are identified via 
$T^v_{\zeta}S^*M = \textup{ker}(d\pi: T_\zeta S^* M \rightarrow T_x M)$ and $T^h_\zeta S^* M = \textup{ker}(\nabla_\zeta: T_\zeta S^*M \rightarrow S^*_xM)$. The metric on $T^h S^*M$ is the pullback metric
induced by $g$ on $M$.  

We define the normal bundle $N\textup{Fix}(G^t_g)$ along $\textup{Fix}(G^t_g)$ to be the orthogonal complement of $T\textup{Fix}(G^t_g)$. Identifying the normal bundle with the quotient induces a definition of $\textbf{P}^\#$ as a linear map 
\[\textbf{P}^\#_N : N\textup{Fix}(G^t_g) \rightarrow N\textup{Fix}(G^t_g).\]
We warn that the derivative of the geodesic flow $dG^t_g$ does not preserve the splitting via the Kaluza-Klein metric. For this reason, we do not use the Kaluza-Klein metric in this paper.

\section{\centering \sc{Proof of Proposition \ref{cleanprop}}}

\subsection{\normalsize  \textbf{Unitary intertwining operator}.}\label{UIO} For a Sunada pair $(M_i, g_i)$ as in the diagram of covers \eqref{maindiag}, there exists an unitary intertwining kernel $\mathcal{A}: L^2(G/H_1) \rightarrow L^2(G/H_2)$ between the isomorphic $G$-modules. The intertwining operator $\mathcal{A}$ is identified as a convolution kernel $A \in \mathbb{C}[H_2\backslash G/H_1])$, which is a function on the double coset space satisfying the identity 
$A(b_2ab_1) = \chi_2(b_2)A(a)\chi_1(b_1)$ for all $b_2\in H_2$, $b_1\in H_1$ [pg. 365, \cite{Ma78}], where $\chi_1$, $\chi_2$ are respectively the 1-dimensional characters of $H_1$, $H_2$.

For each $A$, S. Zelditch constructed a unitary operator given by the averaged weighted sum of Radon transforms
\begin{equation*}
    U_A = \frac{1}{\#H_1}\sum_{a\in G} A(a)\pi_{2*}T_a\pi_1^*: L^2(M_1, dx_1) \rightarrow L^2 (M_2,dx_2)
\end{equation*}
where $T_a$ is the translation by a group element $a$ on the finite covering space $M$ [pg. 709, \cite{Z92}].

We show that $U_A$ lifts to a unitary operator $\widetilde{U}_A$ between the space of $L^2$-distributions on $S_{g_i}^*M$ associated with the diagram of finite normal covers \eqref{2diag} which we repeat for clarity
\begin{equation*}
  \begin{tikzcd}
& S^*M \arrow[dl,swap,"\widetilde{\pi}_1"]  \arrow[dr,"\widetilde{\pi}_2"]\\
S^*(H_1\backslash M)  && S^*(H_2\backslash M)  
\end{tikzcd} 
\end{equation*}

\begin{proposition}\label{Radon}
The finite Radon transform $\widetilde{U}_A: L^2 (S^*M_1, d\mu_{L_1}) \rightarrow L^2 ( S^*M_2, d\mu_{L_2})$ given in \eqref{liftKoop} is a unitary operator mapping the space of $L^2$- distributions on $S^*M_1$ into the space of $L^2$-distributions on $S^*M_2$.
\end{proposition}
\begin{proof}
It is clear that $\pi_i : M \rightarrow M_i= H_i\backslash M$ determines the lifted cover $\widetilde{\pi_i} : S^*M \rightarrow S^*M_i$ given by the dual of the differentials ${d\pi_i}$ restricted to the unit tangent bundle. Thus, the $\pi_i$'s are basewise covers and cover the base maps.

For each $i$, we regard $L^2(S^*M_i)$ as the space $L^2(S^*M)^{H_i}$ consisting of $H_i$-invariants elements of $L^2(S^*M)$ 
since $S^*M_i$ is obtained via quotient the induced action of $H_i$ on $S^*M$. For $f \in L^2(S^*M)$, 
\begin{equation*}
\widetilde{\pi}_{i*}f(x',\xi') 
=\sum_{h\in H_i}(\widetilde{T}_h)^*f(x',\xi')=\sum_{h\in H_i} f(T_h(x),T^*_{h^{-1}}(\xi)).
\end{equation*}
For simple notation, let us write $\widetilde{\pi}_{i*} = \sum_{h\in H_i} \widetilde{T}_{h}$. Consider the product of $\widetilde{U}_A$ with its adjoint, by change of variables, 
\begin{align*}
 \widetilde{U}_A^\dagger\widetilde{U}_A &= \frac{1}{\# H_1\# H_2}\sum_{a,b \in G}\overline{A}(b^{-1})A(a)(\widetilde{\pi_1}_{*}\widetilde{T}_{b^{-1}}\widetilde{\pi_2}^*)(\widetilde{\pi_2}_{*}\widetilde{T}_{a}\widetilde{\pi_1}^*)\\
                            &=\frac{1}{\# H_1\# H_2}\sum_{a,b \in G}\overline{A}(b^{-1})A(a)[(\sum_{b_1\in H_1} \widetilde{T}_{b_1}\widetilde{T}_{b^{-1}})(\sum_{b_2\in H_2}\widetilde{T}_{b_2}\widetilde{T}_{a})]\\
                            &=\frac{1}{\# H_1\# H_2}\sum_{a,b\in G, b_1\in H_1, b_2\in H_2}\overline{A}(b^{-1})A(a)\widetilde T_{b_1b^{-1}b_2}\widetilde{T}_a\\ &=\frac{1}{\# H_1\# H_2}\sum_{b_1,b_2}\sum_{a,b}\overline{A}(b^{-1})A(a)\widetilde T_{\overline{b^{-1}}} \widetilde{T}_a    \text{~~~~~~~~~~~~~~~~~~~~~$(\overline{b^{-1}}:= b_1b^{-1}b_2)$}\\
                            &=\frac{1}{\# H_1\# H_2}\sum_{b_1,b_2}\sum_{a,b}\overline{A}(b^{-1})A(b^{-1}a)\widetilde T_{b^{-1}\overline{b^{-1}}a}\\
                            &=\frac{1}{\# H_1\# H_2}\sum_{b_2}\sum_{b_1}\sum_{a}A^*A(a)\widetilde{T}_{a} \text{~~~~~~~~~~~~~~~~~~~~~~~~~~~~~~$(a\mapsto b^{-1}\overline{b^{-1}}a)$} \stepcounter{equation}\\
                            &=\frac{1}{\# H_2}\sum_{b_2}\frac{1}{\# H_1}\sum_{a}\chi_{H_1}(a)\widetilde{T}_a \text{~~~~~~~~~~~~~~~~~~~~~~~~~~~~~~~~~~$(A$ is unitary)}\\ 
                            &=\frac{1}{\# H_1}\sum_{b_1}\widetilde{T}_{b_1} = \textup{Id}_{L^2(S^*M)^{H_1}}.
\end{align*}

\end{proof}


\subsection{ \normalsize \textbf{Gelfand-Leray form on a clean fixed Lagrangian component.}}\label{gelfandlerayformsection} We compute  $\Pi_*\Delta^*k^\bullet$ for the general case of clean fixed point sets.

Let $Z$ be a clean component of $\textup{Fix}(G^\tau_g)$. Let $\zeta$ be a point in $Z$ and $\gamma \in \mathscr{G}$ be a $\tau$-periodic orbit starting at $\zeta$.  Let $\zeta_1, \ldots, \zeta_{2n-1}$ be coordinates about $\zeta$ from which the normal bundle $NZ$ is described by $\zeta_1 = \ldots = \zeta_k =0$ with $k = \dim(Z)$, and the $\zeta_i$'s are defined by 
\begin{equation}\label{loccoord}
    \xi_1 - (G_g^\tau \xi)_1 = \zeta_1, \ldots , \xi_{k} - (G_g^\tau \xi)_{k} = \zeta_{k}, \xi_{k+1} = \zeta_{k+1},\ldots, \xi_{2n-1} = \zeta_{2n-1}. 
\end{equation}
Recall that 
the pullback under $\rho$ of a 1-form $\beta$ in $T^*S_g^*M$ supported on the graph of $G^\tau_g$ is 
\[((-dG^\tau_g)^*\beta, \beta) \in \rho^*T^*\big|_{\textup{Graph}(G^\tau_g)}S_g^*M,\] 
which is further identified with 
\[((-(dG^\tau_g)^\#)^*\beta, \beta) \xrightarrow[]{d\Delta^*}((I - dG^\tau_g )^\#)^* \beta \]
as an element of $\Delta^*\rho^* T^*\big|_{Z} \cong N^*Z$ with foot-points in $Z$.   
Since $|NZ| \otimes |T|Z = |T|S_g^*M$, the induced $\delta$-distribution on $Z$ is the Leray form
\begin{align*}
  d\nu_{Z}(\zeta) :
  &= \frac{d\mu_L(\zeta)}{((I - dG^t_g)^\#)^*|d\zeta_{k+1}\ldots d\zeta_{2n-1}|(\zeta)}\\
    &= \frac{d\mu_L(\zeta)}{|\det \mathbf{J}^\perp \big|_\zeta (I - G^\tau_g)||d\zeta_{k+1}\ldots d\zeta_{2n-1}|(\zeta)|}.
\end{align*}
With respect to the basis $\{\partial_{\zeta_{k +1} }\big|_\zeta, \ldots, \partial_{\zeta_{2n-1}}\big|_\zeta\}$, the transversal Jacobian is the matrix 
\begin{equation}\label{normalJacobian}
    \mathbf{J}^\perp \big|_\zeta (I - G^\tau_g)  = \begin{pmatrix}
\partial_{\zeta_{k+1}}\big|_\zeta(I - G^\tau_g)_{k+1} & \ldots & \partial_{\zeta_{k +1}}\big|_\zeta(I - G^\tau_g)_{\zeta_{2n-1}}\\
\vdots & & \vdots \\
\partial_{\zeta_{2n-1}}\big|_\zeta(I - G^\tau_g)_{k +1} & \ldots & \partial_{\zeta_{2n-1}}\big|_\zeta(I - G^\tau_g)_{2n-1}
\end{pmatrix}.
\end{equation}
Rewriting the row vectors in \eqref{normalJacobian} in terms of an orthonormal basis of the space of periodic $Z$-transversal Jacobi fields along $\gamma$ (\S \ref{canvolZj}) to obtain  
\[
 d\nu_Z = \frac{d\mu_L }{|\det(I - \textbf{P}^\#_Z)||d\zeta_{k+1}\ldots d\zeta_{2n-1}|}.
\] 
By integrating along the fibre $\Pi^{-1}(\zeta)$, it follows that
\begin{equation}\label{guilleminproof}
    \Pi_*\Delta^*k^\bullet = \int_{\Pi^{-1}(\textup{Fix}(G^\bullet_g))}|\Delta^*k^t\wedge dt|
= \sum_{\tau\in \textup{Lsp}(S^*M)}\delta(t- \tau) \sum_{Z \subset \textup{Fix}G^\tau_g} \int_Z d\nu_Z.
\end{equation}

\textbf{Remark.} If $\gamma$ is a non-degenerate orbit, we note that $\frac{d\mu_L}{|\dot{\gamma}|}$ is a volume element on the local transversal hypersurface $\Sigma$ to $\gamma$. Equivalently, this is the symplectic form $\dot{\gamma}\lrcorner d\mu_L$. As a result, 
\[
\Delta^*k^t\big|_\gamma= \frac{d\mu_L}{(I - dG^\tau_g)^* (\dot{\gamma}\lrcorner d\mu_L )} = \frac{|\dot{\gamma}|}{|\det(I - \textbf{P}_\gamma)|} \in |T|\gamma.
\]
In the presence of high-dimensional fixed point sets, the dimension of the normal space to the fixed point set decreases - hence the modified version of the Poincare map in the case of Bott-Morse non-degeneracy.


\subsection{\normalsize \textbf{Flat-trace distribution defined by oscillatory integral}.}\label{mollification} Recall that the flat-trace kernel at a fixed time $t$ is the $\delta$-distribution $\delta(\xi - G^t_g \xi)$.
Let us fix a finite smooth atlas on $S^*M$ with coordinate charts $(\Omega_i,\varphi_i)$ where the $\varphi_i$'s are local diffeomorphisms. Within these charts, $(\xi_1, \ldots, \xi_{N=2n-1})$ are introduced as in \eqref{loccoord}
where $\zeta_1, \ldots, \zeta_K$ represents local coordinates on a clean submanifold $Z$ associated with a length-spectral time, and $K $ is the dimension of $Z$.  

Consider a test function $\varrho \in \mathcal{D}(S^*M)$. Let $\{\rho_i \in C_c^\infty(\Omega_i; [0,1])\} $ be a corresponding partition of unity on $S^*M$ so that by the Fourier integral theorem, 

\begin{equation*}
        \varrho(\xi) = \sum_i \int_{\R^N}\int_{\varphi_i(\Omega_i)}\varphi_{i*}(\rho_i\varrho) (y) e^{i \langle  \theta, \varphi_i(\xi) - y\rangle } dy d\theta.
\end{equation*}
Also, for $\tau \in \textup{Lsp}(S^*M)$, $\xi \in \textit{supp}(\Omega_i)$, the flat-trace kernel of $V^\tau_g$ can be written as 
\begin{equation}\label{Fourier}\delta (\xi - G_g^\tau \xi) = (2\pi)^{-\frac{N - K}{2}}\int_{\R^N}\ e^{i\langle \theta  , \varphi_i(\xi) - \varphi_i(G^\tau_g\xi)  \rangle } d\theta. \end{equation}
in the sense of distributions. Pairing the kernel with $\varrho$ gives
\begin{align*}\left(k^\tau, \Delta_*\varrho \right) & =  \sum_i \int_{S^*M}\int_{\R^N} \int_{\varphi_i(\Omega_i)} \varphi_{i*}(\rho_i\varrho) (y) e^{i\langle \theta, \varphi_i(\xi) - y \rangle} \delta(\xi - G^\tau_g \xi) dyd\theta  d\mu_L(\xi).\end{align*}
The Guillemin flat-trace distribution at time $t =\tau$ has a global representation given by  
\begin{equation}\label{Fourier2}
    \left( \Pi_*\Delta^*k^\bullet\big|_\tau , \varrho \right) = (2\pi)^{-\frac{N-K}{2}} \sum_{i} \int_{\Omega_i} \int_{\R^N} \rho_i \varrho(G^\tau_g \xi)  e^{i\langle \theta , \varphi_i (\xi) - \varphi_i(G^\tau_g \xi) \rangle }d\theta d\mu_L(\xi).
\end{equation}

We define the change of variables $\kappa: S^*M \rightarrow \R^N$ piecewise-ly as follows 
$$\begin{cases}
    \xi \mapsto \varphi_i(\xi)  & \text{~ for~} \xi \in \Omega_i \cap Z^c \\
    \zeta \mapsto (0, \ldots, 0, \zeta_{K+1}, \ldots, \zeta_N) & \text{~ for~} \zeta \in Z = \text{Fix}G^\tau_g
\end{cases}.$$
The phase function $\phi : S^*M \times S^*M \times \R^N \rightarrow \R$ is defined by 
\begin{equation}\label{thephasefunction}
    \phi(\xi, \eta, \theta) : = \langle \kappa(\xi) - \kappa(\eta), \theta \rangle.
\end{equation}
Moreover, there exists a neighborhood $\Omega$ of the diagonal of $S^*M\times S^*M$ on which we can find a $C^\infty$-map $\Psi: \Omega \rightarrow GL(N, \R)$ satisfying 
\begin{equation}\label{cancelsdkappadxi}
    \phi(\xi, G^\tau_g\xi, \Psi(\xi, G^\tau_g\xi)\theta) = \langle \kappa(\xi) - \kappa(G^\tau_g \xi), \theta  \rangle
\end{equation} and 
\begin{equation}\label{det=1}
\det \Psi(\xi, G^\tau_g\xi) \det D^2_{\theta, \xi}\phi (\xi, G^\tau_g \xi, \theta)\big|_{\xi = G^\tau_g \xi} = 1.
\end{equation}
Let $\bar{\phi} (\xi, \theta; \tau) : = \phi(\xi, G^\tau_g\xi, \theta)$ to be the restriction of the phase function onto $S^*M \times \R^N$. The pairing at \eqref{Fourier2} becomes 



\begin{equation}\label{oscillatoryintegral1}
(2\pi)^{-\frac{N - K}{2}}\int_0^\infty r^{(N-K) -1} \int_{\mathbb{S}^{N-1}} \int_{S^*M} \varrho (\xi) e^{i r \bar{\phi} (\xi, \bar{\theta}; \tau)} |\det(D_\xi \kappa)|   d\xi d\sigma (\bar{\theta}) dr.
\end{equation}
Integrating by parts with respect to the radial variable produces an integrand of the form 
\begin{equation}\label{innerprodwithsphereint}
J_\tau(r): =  \int_{\mathbb{S}^{N-1}} \int_{S^*M}\varphi(\xi) \bar{\phi}(\xi, \bar{\theta}; \tau)   e^{i r \bar{\phi}(\xi, \bar{\theta}; \tau) }d\xi d\sigma(\bar{\theta})
\end{equation}
where $\varphi$ is a scalar function with compact support on $S^*M$ depending on $\kappa, \varrho$. By using coordinate charts and a finite partition of unity on $\mathbb{S}^{N-1}$, we can write \eqref{innerprodwithsphereint} as 
$$
 \int_{\R^{N-1}} \int_{S^*M} a(\theta') \varphi(\xi)\check{\phi}(\xi, \theta'; \tau) e^{i r \check{\phi}(\xi, \theta'; \tau)}d\xi d\theta' 
$$
with $a\in C^\infty_c(\R^{N-1}\backslash\{0\})$ supported near the origin with $D_{\theta'}a  \neq 0$ and $\check{\phi}$ is the phase function set by 

$$\check{\phi}(\xi, \theta'; \tau) = (\xi' - (G^\tau_g \xi)' ) \cdot \theta' + ( \xi_N - (G^\tau_g\xi)_N) (1 - |\theta'|^2)^{\frac{1}{2}}$$ 
where $\kappa(\xi) - \kappa(G^\tau_g \xi)  = (\xi' - (G^\tau_g \xi)', \xi_N - (G^\tau_g\xi)_N) \in \R^{N-1} \times \R$. With $\xi' = (\xi_1, \ldots, \xi_{N-1})$, $\theta' = (\theta'_1,\ldots, \theta'_{N-1})$, we have 
$$
\partial_{\xi_j} \check{\phi} = \sum_{i=1}^{N-1} \theta'_i(\delta^i_j  - \partial_{\xi_j}(G^\tau_g\xi)_i ) - \partial_{\xi_j}(G^\tau_g \xi)_N (1 - |\theta'|^2)^{\frac{1}{2}} \text{~ for $j = 1,\ldots, N-1$},
$$
$$
\partial_{\xi_N} \check{\phi} = - \theta' \cdot D_{\xi_N}(G^\tau_g \xi)' + (1- \partial_{\xi_N}(G^\tau_g \xi)_N) (1 - |\theta'|^2)^{\frac{1}{2}}, 
$$
$$
\partial_{\theta'} \check{\phi} = \xi' - G^\tau_g \xi' - \frac{\xi_N - (G^\tau_g \xi)_N}{(1 - |\theta'|^2)^{\frac{1}{2}}}\theta'
$$
which vanish at $\theta' = 0$ when $ G^\tau_g \xi  = \xi $. Furthermore, 
$$D^2_{\xi' \xi '} \check{\phi} = -\theta' \cdot D^2_{\xi' \xi '} (G^\tau_g\xi)' - D^2_{\xi' \xi '} (G^\tau_g \xi)_N(1 - |\theta'|^2)^{\frac{1}{2}} ,$$
$$ \partial^2_{\xi_j\theta'_k} \check{\phi} = \sum_{i = 1}^{N-1} \delta^i_k(\delta^i_j - \partial_{\xi_j}(G^\tau_g\xi)_i) + \frac{\partial_{\xi_j}(G^\tau_g\xi)_N}{(1 - |\theta'|^2)^{\frac{1}{2}}}\theta'_k \text{~ for $k = 1,\ldots, N-1$},$$
$$D^2_{\theta'\theta'}\check{\phi} = (G^\tau_g\xi)_N - \xi_N .$$ 
Let $(\zeta, 0)$ be a critical point, $\zeta \in \textup{Fix}G^\tau_g$. The normal Hessian of $\check{\phi}$ at $(\zeta, 0)$ is the $2(N-K) -1 \times 2(N- K) -1$ - matrix 
\[ \text{Hess}^\perp\big|_{(\zeta,0)}\check{\phi}  = \begin{pmatrix}\mathbf{D^2}^\perp_{\xi\xi}\check{\phi} \big|_{(\zeta, 0)} & \mathbf{D^2}^\perp_{\xi\theta'}\check{\phi}\big|_{(\zeta, 0)} \\&\\ \mathbf{D^2}^\perp_{ \theta'\xi}\check{\phi} \big|_{(\zeta, 0)}& \mathbf{D^2}^\perp_{\theta'\theta'}\check{\phi}\big|_{(\zeta, 0)}\end{pmatrix}\]
where 
\[\mathbf{D^2}^\perp_{\xi\xi}\check{\phi} \big|_{(\zeta, 0)} = \begin{pmatrix}
-\partial^2_{\xi_j\xi_k}\big|_{\zeta}(G^\tau_g\xi)_N    
\end{pmatrix}_{N - K \times N - K },\]

\[ \mathbf{D^2}^\perp_{\theta' \xi}\check{\phi} \big|_{(\zeta, 0)} = \mathbf{D^2}^\perp_{\xi\theta'}\check{\phi}^\top \big|_{(\zeta, 0)} =  \begin{pmatrix}
    \delta^k_j - \partial_{\xi_k}\big|_{\zeta}(G^\tau_g \xi)_j
\end{pmatrix}_{N - K -1 \times N - K}, \]

\[
\mathbf{D^2}^\perp_{\theta'\theta'}\check{\phi} \big|_{(\zeta, 0)}= \mathbf{O}_{N- K-1 \times N - K-1}.
\]
By Schur-complement formula, $\left | \det \text{Hess}^\perp\big|_{(\zeta, 0)} \check{\phi} \right | $ equates to 
\begin{equation}\label{detofhessianofcheckphi}
\left| \det \mathbf{D^2}^\perp_{\xi\xi}\check{\phi}\big|_{(\zeta, 0)}  \right| \left | \det \left ( \mathbf{D^2}^\perp_{\theta' \xi}\check{\phi}\big|_{(\zeta, 0)} \cdot \mathbf{D^2}^\perp_{\xi\xi}\check{\phi} ^{-1}\big|_{(\zeta, 0)} \cdot \mathbf{D^2}^\perp_{\xi\theta'}\check{\phi}\big|_{(\zeta, 0)} \right) \right |.
\end{equation}

\begin{proposition}\label{Jterm}
\[J_\tau(r)  = O(r^{-(N - K +\frac{3}{2})}) \text{~~ for $r \gg 1 $}.\]
\end{proposition}

\begin{proof}

Let $u(\xi, \theta') :  = a(\theta')\varphi(\xi)\check{\phi}(\xi,\theta';\tau) \in C^{2k}_c(\text{spt}(a))$, $k \geq 3$. We apply the method of stationary phase in the $2(N - K) -1$ normal dimensions to obtain the estimation [cf. Theorem 7.7.5, \cite{Hor}] 
\begin{equation}\label{Hormanderstationaryphase}
\left| J_\tau(r) - \int_Z \frac{e^{ir\check{\phi}(\zeta, 0 ;\tau)}}{\det((2\pi i)^{-1}r\text{Hess}^\perp\big|_{(\zeta,0)}\check{\phi})^{\frac{1}{2}}}d\zeta . \sum_{m < k}r^{-m}L_mu \right| \leq Cr^{-k}\sum_{|\alpha| \leq 2k } \sup|D^{\alpha}u|
\end{equation}
for a constant $C$, 
\begin{equation}\label{L_mu}
L_mu = \sum_{\nu-\mu = m}\sum_{2\nu\geq 3\mu} i^{-l}2^{-\nu}\langle \nabla^2 \check{\phi}^{-1}\big|_{(\zeta,0)} D , D \rangle^\nu (h^\mu_{(\zeta,0)}u)(\zeta,0) / \mu!\nu!\end{equation}
where
\begin{align*}
h_{(\zeta, 0)} (\xi, \theta') &= \check{\phi}(\xi, \theta'; \tau) -\frac{1}{2}\sum_{i,j = 1}^N \partial^2_{\xi_i\xi_j}\big|_{\zeta}(G^\tau_g\xi)_N (\xi_i - \zeta_i)(\xi_j - \zeta_j) \\
&+ \sum_{i=1}^N\sum_{j=1}^{N-1}\left(\delta^i_j - \partial_{\xi_i}\big|_{\zeta}(G^\tau_g \xi)_j \right)(\xi_i - \zeta_i)\theta'_j\end{align*}
and \begin{align*}
    \langle \nabla^2 \check{\phi}^{-1}\big|_{(\zeta,0)} D, D  \rangle  &= \frac{1}{\left | \det \text{Hess}^\perp\big|_{(\zeta, 0)} \check{\phi} \right |} \bigg( -\sum_{i,j =1}^N \partial^2_{\xi_i\xi_j}\big|_{\zeta}(G^\tau_g\xi)_N \partial^2_{\xi_i\xi_j}\\
    &+ 2 \sum_{i,j =1 }^{N-1} \left(\delta^i_j - \partial_{\xi_j}\big|_{\zeta}(G^\tau_g \xi)_i\right) \partial^2_{\theta'_i\xi_j} - \partial_{\xi_N}\big|_{\zeta}(G^\tau_g \xi)_{N-1}\partial^2_{\xi_N\theta'_{N-1}} \bigg).\\
\end{align*}
The coefficients $L_mu$ vanish up to the second order. Indeed, we first observe that $L_0u = 0$, 
\begin{equation}\label{L_1u}
L_1u = \sum_{(\mu,\nu) \in \{(0,1),(1,2),(2,3)\}} i^{-m}2^{-\nu}\langle \nabla^2 \check{\phi}^{-1} \big|_{(\zeta,0)}D , D \rangle^\nu (h^\mu_{(\zeta,0)}u)(\zeta,0) / \mu!\nu!. \end{equation}
Since
\begin{align*}
\partial^2_{\theta'_i\xi_j}u(\zeta, 0) &= a'_i(0) \left(\check{\phi}(\zeta, 0;\tau)\varphi'_{\xi_j}(\zeta) + \varphi(\zeta)\partial_{\xi_j}\check{\phi}(\zeta, 0;\tau)\right) \\
&+a(0)\left( \varphi'_{\xi_j}(\zeta)\partial_{\theta'_i}\check{\phi}(\zeta,0;\tau) + \varphi(\zeta)\partial^2_{\theta'_i\xi_j} \check{\phi}(\zeta,0;\tau) \right) = 0
\end{align*}
for $i = 1, \ldots, N-1$, $j = 1, \ldots, N$,
\[\partial^2_{\xi_i\xi_j}u(\zeta, 0) = a(0)\partial^2_{\xi_i\xi_j}\varphi(\zeta)\check{\phi} (\zeta, 0; \tau)= 0, \]
the summand indexed by $(0,1)$ in \eqref{L_1u}  vanishes. For $(\mu,\nu) = (1,2)$, the fourth-order partial derivatives 
$$\partial^4_{\xi_{i,j,k,l}}h_{(\zeta,0)}u = a \partial^4_{\xi_{i,j,k,l}}\varphi \check{\phi}h_{(\zeta, 0)},$$
\begin{align*}
    \partial^4_{\theta'_i\xi_j\theta'_k\xi_l}h_{(\zeta,0)}u &= a\big( \varphi\partial^4_{\theta'_i\xi_j\theta'_k\xi_l} +     \varphi'_j\partial^3_{\theta'_i\xi_l\theta'_k} 
    + \varphi''_{\xi_{j,l}}\partial^2_{\theta'_{i,k}} + \varphi'_l\partial^3_{\theta'_i\xi_j\theta'_k}  \big)  \check{\phi}h_{(\zeta, 0)}\\  &+a'_i\big(\partial^2_{\xi_{j,l}}\varphi\partial_{\theta'_k} +
    \varphi'_l\partial^2_{\xi_j\theta'_k} + \varphi'_j\partial^2_{\xi_l\theta'_k} + \varphi\partial^3_{\xi_j\theta'_k\xi_l}  \big)\check{\phi}h_{(\zeta, 0)} \\
    &+ a''_{i,k}\partial^2_{\xi_{j,l}}\varphi \check{\phi}h_{(\zeta, 0)},
\end{align*}
\begin{align*}
    \partial^4_{\xi_i\xi_j\theta'_k\xi_l}h_{(\zeta,0)}u &= a \big(  \varphi'''_{i,j,l}\partial_{\theta'_k} + \varphi''_{j,l}\partial^2_{\xi_i\theta'_k} + \varphi'_l\partial^3_{\xi_{i,j}\theta'_k} \\
    &+\varphi'_i\partial^3_{\xi_j\theta'_k\xi_l}  + \varphi''_{i,l}\partial^2_{\xi_j\theta'_k} + \varphi\partial^4_{\xi_{i,j}\theta'_k\xi_l}\big)\check{\phi}h_{(\zeta, 0)} \\
    &+ a'_k\partial^3_{\xi_{i,j,l}}\check{\phi}h_{(\zeta, 0)} 
\end{align*}
with $1 \leq j,l \leq N$, $1 \leq i,k \leq N-1$
contain factors of zeroth order of $a$, zeroth and first orders of $\check{\phi}$ and $h_{(\zeta,0)}$, which vanish at $(\zeta, 0)$. Hence, 
$$
\langle \nabla^2 \check{\phi}^{-1}\big|_{(\zeta,0)} D , D \rangle^2 (h_{(\zeta,0)}u)(\zeta,0)= 0.
$$
For $(\mu,\nu) = (2,3)$, the sixth orders are  
\[\partial^6_{\xi_{i,j,p,q,m,n}}(h^2_{(\zeta,0)}u)  = a \partial^6_{\xi_{i,j,p,q,m,n}}(\varphi\check{\phi}h^2_{(\zeta,0)}),\]
\begin{align*}
    \partial^6_{\xi_{i,j,p,q}\theta'_k\xi_l}(h^2_{(\zeta,0)}u) &= a \big(\partial^4_{\xi_{i,j,p,q}}(\varphi'_l\partial_{\theta'_k}) + \partial^4_{\xi_{i,j,p,q}}(\varphi\partial^2_{\theta'_k\xi_l}) \big) \check{\phi}h^2_{(\zeta,0)} \\
    & + a'_k \partial^4_{\xi_{i,j,p,q}} \varphi'_l \check{\phi}h^2_{(\zeta,0)}, 
\end{align*}
\begin{align*}
     \partial^6_{\xi_{i,j}\theta'_k\xi_l\theta'_p\xi_q}(h^2_{(\zeta,0)}u) &= \partial^2_{\xi_{i,j}} \big( a(\varphi'_l\partial^3_{\theta'_{k,p}\xi_q}+ \varphi\partial^4_{\theta'_k\xi_l\theta'_p\xi_q}+ \varphi''_{l,q}\partial^2_{\theta'_{k,p}} + \varphi'_q\partial^3_{\theta'_k\xi_l\theta'_p}) \\
&+a'_p(\varphi'_l\partial^2_{\theta'_k\xi_q}+\varphi\partial^3_{\theta'_k\xi_{l,q}}+ \varphi''_{l,q}\partial_{\theta'_k}+\varphi'_q\partial^2_{\theta'_k\xi_l})\\
& + a'_k(\varphi_l\partial^2_{\theta'_p\xi_q}+ \varphi\partial^3_{\xi_l\theta'_p\xi_q} + \varphi''_{l,q}\partial_{\theta'_p}+ \varphi'_q\partial^2_{\xi_l\theta'_p}) \\
& + a''_{k,p}(\varphi'_l \partial_{\xi_q} + \varphi\partial^2_{\xi_{l,q}} + \varphi''_{l,q} + \varphi'_q\partial_{\xi_l})\big) \check{\phi}h^2_{(\zeta,0)}
\end{align*}
with $\partial^2_{\xi_{i,j}} (a\varphi\partial^4_{\theta'_k\xi_l\theta'_p\xi_q}\check{\phi}h^2_{(\zeta,0)}) = a\partial^2_{\xi_{i,j}}(\varphi \partial^4_{\theta'_k\xi_l\theta'_p\xi_q}\check{\phi}h^2_{(\zeta,0)})$ and $1\leq i,j,l,q \leq N, ~ 1 \leq k,p \leq N-1$, 
\begin{align*}
\partial^6_{\theta'_k\xi_l\theta'_p\xi_q\theta'_r\xi_s}(h^2_{(\zeta,0)}u) & = \partial^2_{\theta'_k\xi_l}\big( a (\varphi''_{s,q}\partial^2_{\theta'_{r,p}} + \varphi'_s\partial^3_{\theta'_p\xi_q\theta'_r}+ \varphi'_q\partial^3_{\theta'_{r,p}\xi_s} +\varphi\partial^4_{\theta'_p\xi_q\theta'_r\xi_s})  \\
& + a'_r(\varphi_q\partial^2_{\theta'_p\xi_s} + \varphi\partial^3_{\theta'_p\xi_{q,s}} + \varphi''_{q,s}\partial_{\theta'_p} + \varphi'_s\partial^2_{\theta'_p\xi_q}) \\
&  + a'_p(\varphi''_{s,q}\partial_{\theta'_r} + \varphi'_s\partial^2_{\xi_q\theta'_r}+ \varphi'_q\partial^3_{\theta'_r\xi_s} + \varphi\partial^3_{\xi_q\theta'_r\xi_s}) \\ 
& + a''_{r,p}(\varphi'_q\partial_{\xi_s} + \varphi\partial^2_{\xi_{q,s}} + \varphi''_{q,s} + \varphi'_{s}\partial_{\xi_q}) \big) \check{\phi}h^2_{(\zeta,0)}
\end{align*}
with 
\begin{align*}\partial^2_{\theta'_k\xi_l} (a\varphi\partial^4_{\theta'_p\xi_q\theta'_r\xi_s}\check{\phi}h_{(\zeta,0)}) &= \big ( a (\varphi'_l \partial^5_{\theta'_k\xi_q\theta'_r\xi_s}+ \varphi \partial^6_{\theta_k\xi_l\theta_p\xi_q\theta'_r\xi_s}) \\
&+  a'_k (\varphi'_l\partial^4_{\theta'_p\xi_q\theta'_r\xi_s} + \varphi \partial^5_{\xi_l\theta'_p\xi_q\theta'_r\xi_s}) \big) \check{\phi}h^2_{(\zeta,0)}
\end{align*}
and $1 \leq l,q,s \leq N,~ 1 \leq k,p,r \leq N-1 $. Notice that any less-than-six partial orders of $\check{\phi}h^2_{(\zeta, 0)}$ contain factors of zeroth and/or first orders of $\check{\phi}$ and $h_{(\zeta,0)}$, which implies  
\[
\langle \nabla^2 \check{\phi}^{-1}\big|_{(\zeta,0)} D , D \rangle^3 (h^2_{(\zeta,0)}u)(\zeta,0)= 0
\]
and hence $L_1 u = 0 $. 

When $m=2$, in the case of $(\mu,\nu) = (0,2)$, respectively $(1,3)$, terms of the form 
\[a'_k\varphi\partial^3_{\xi_{i,j,l}}\check{\phi}, ~ a'_k\varphi'_l\partial^2_{\xi_{i,j}}\check{\phi}\partial^2_{\xi_{p,q}}h_{(\zeta,0)} \]
are non-vanishing on the clean submanifold and hence $L_2 u$ is non-zero. From \eqref{detofhessianofcheckphi}, \eqref{Hormanderstationaryphase}, 
\begin{align*}
J_\tau(r) & = (2\pi i r^{-1})^{N - K - \frac{1}{2}}\int_{Z}\frac{d\zeta}{( \det \text{Hess}^\perp\big|_{(\zeta, 0)} \check{\phi})^{\frac{1}{2}} } ~ O(r^{-2})= O (r^{-(N-K) -\frac{3}{2}})              
\end{align*}
as $r \rightarrow \infty$.

\end{proof}


\subsection{\normalsize \textbf{Proof of Proposition \ref{cleanprop}}. }\label{proofcleanprop}
We provide a proof of Proposition \ref{cleanprop} using the method of stationary phase in \S \ref{stationaryphase}. The justification of \eqref{sumflattrace} following the same outline as in [Theorem 8, \cite{G77}] is given in \S \ref{gelfandlerayformsection}.
\begin{proof}
Let $\varrho \in \mathcal{D}(S^*M)$ be a test function and $\gamma$ be a $\tau$-periodic trajectory starting at a point $\zeta$ on a clean component $Z\subset \text{Fix}G^\tau_g$ of dimension $K$. $\zeta_1, \ldots, \zeta_N$ are coordinates about $\zeta$ on $S^*M$ as described in \eqref{loccoord}. 

The orthonormal basis for $T_{\zeta}S^*_gM/ T_{\zeta}Z$ is given by $\{\partial_{\zeta_{K+1}}\big|_\zeta,\ldots, \partial_{\zeta_N}\big|_\zeta\}$. Consider the phase function $\phi$ given in \eqref{thephasefunction}, the critical submanifold of $\phi$ is the joint space 
$$
\text{Crit}(\phi) = \{(\zeta, \bar{\theta}): \bar{\theta} =(\theta_1,\ldots, \theta_K,  0,\ldots, 0) \} \subset Z\times \R^N.
$$
The transverse Hessian of $\phi$  with respect to the basis of the normal quotient space is the block-matrix 
\[
\text{Hess}^\perp \big|_{(\zeta, \bar{\theta})} \phi   = \begin{pmatrix}
\mathbf{D^2}^\perp_{\xi\xi} \big|_{(\zeta, \bar{\theta})} \phi  & \mathbf{D^2}^\perp_{\xi\theta} \big|_{(\zeta, \bar{\theta})}\phi \\
&\\
\mathbf{D^2}^\perp_{\theta\xi}  \big|_{(\zeta, \bar{\theta})} \phi & \mathbf{D^2}^\perp_{\theta\theta} \big|_{(\zeta, \bar{\theta})} \phi
\end{pmatrix}
\]
where 
\begin{align*}
\mathbf{D^2}^\perp_{\xi\theta} \big|_{(\zeta, \bar{\theta})} \phi = \mathbf{D^2}^\perp_{\theta\xi}\big|_{(\zeta, \bar{\theta})} \phi 
&= \begin{pmatrix}
\delta^l_j - \partial_{\zeta_j}\big|_{\zeta}(G^\tau_g \xi)_l \\
\end{pmatrix}_{N - (K+1) \times N -(K+1)} , 
\end{align*}

\[ \mathbf{D^2}^\perp_{\xi\xi} \big|_{(\zeta, \bar{\theta})} \phi = 
\begin{pmatrix}
    - \theta \cdot D^2_{\zeta_j\zeta_l}\big|_{\zeta} G^\tau_g \xi
\end{pmatrix} =  \mathbf{O}_{N -(K+1) \times N -(K+1)} = 
\mathbf{D^2}^\perp_{\theta\theta} \big|_{(\zeta, \bar{\theta})} \phi.  \]
The diagonal blocks vanish due to geodesicity of $\gamma$ and linearity of $\phi$ in $\theta$. The off-diagonal sub-matrices with respect to an orthonormal basis $U_i,V_i$ of $\J^\perp_Z(\gamma)$ corresponding to $(\ast)$ as in \eqref{Jbasis} are precisely 
\[\begin{pmatrix}
    (\delta^i_l - \langle \partial_{\zeta_i}\big|_{\zeta}, U_l(\tau)\rangle)_{i,l} & (\langle \partial_{\zeta_i}\big|_{\zeta}, V_l(\tau) \rangle)_{i,l}\\
    & \\
    (\langle\partial_{\zeta_i}\big|_{\zeta}, \nabla_{\dot\gamma}U_l(\tau)\rangle)_{i,l} & (\delta^i_l - \langle \partial_{\zeta_i}\big|_{\zeta}, \nabla_{\dot\gamma}V_l(\tau) \rangle)_{i,l}
    \end{pmatrix}.\]
\medskip
since the derivatives $\{\partial_{\zeta_i} G^t_g \}_{K+1 \leq i \leq N}$ of the geodesic are Jacobi fields. 
This implies 
\[|\det \text{Hess}^\perp\big|_{(\zeta, \bar{\theta})} \phi| = |\det (I - \textbf{P}_Z^\#)_\zeta|^2\]
for all $\bar{\theta}$. By applying Proposition \ref{Jterm} and the joint-normal stationary phase in light of Lemma \ref{normalstatlemma}, integration by parts in \eqref{oscillatoryintegral1} yields
\begin{align*}\label{semiclassicalinr}
 \left( \Pi_*\Delta^*k^\bullet\big|_\tau , \varrho \right) 
& =\lim_{r \rightarrow \infty} \frac{r^{N-K}}{(2\pi)^{\frac{N-K}{2}}}
    \int_{S^*M}\int_{\R^N}\varrho(\xi) e^{ir\phi(\xi, G^\tau_g \xi, \bar{\theta})} d\mu_L(\xi)d\bar{\theta} -i\int_0^\infty r^{N-K}J_\tau(r) dr \\
&= \int_Z \varrho (\zeta) \frac{d\mu_L(\zeta)}{|\det\text{Hess}^\perp\big|_{(\zeta, \bar{\theta})} \phi|^{\frac{1}{2}}|d\zeta_{K+1} \ldots d\zeta_N|}  - i\int_{\{r\gg 1\}} r^{N-K}J_\tau(r)  dr\\
& = \int_Z  \varrho (\zeta) \frac{ |d\textup{Vol}_{can}(Z)|(\zeta)}{|\det (I - \textbf{P}_Z^\#)_\zeta|}. 
\end{align*}
Consequently, for any test function $f \in C^\infty_c((0,\infty))$, we have 
\begin{align*}
    \left(\textup{Tr}^\flat V^\bullet_g, f\right) &= \sum_{\tau \in \textup{Lsp}S^*_gM} f(\tau) \int_{S^*M} \delta(\xi - G^\tau_g\xi) d\mu_L(\xi)\\
    &= \sum_{\tau \in \textup{Lsp}S^*_gM} f(\tau)  \int_Z \frac{ |d\textup{Vol}_{can}(Z)|(\zeta)}{|\det (I - \textbf{P}_Z^\#)_\zeta|}.
\end{align*}
This implies 
\[\textup{Tr}^\flat V^t_g   =  \sum_{\tau \in \textup{Lsp}S^*_gM}\sum_{Z \subset \textup{Fix}G^\tau_g}   \delta(t-\tau) \int_Z \frac{ |d\textup{Vol}_{can}(Z)|(\zeta)}{|\det (I - \textbf{P}_Z^\#)_\zeta|}.    \]

\end{proof}



\section{\centering \sc{Proof of theorem 1.1}}\label{proofofsameflattrace}
\begin{proof}Recall the Radon transform $\widetilde{U}: L^2(S^*M_1) \rightarrow L^2(S^*M_2)$ and its adjoint inverse 
$$
\widetilde{U} = \frac{1}{\#H_1}\sum_{a\in G}A(a) \widetilde{\pi_2}_*T_a\widetilde{\pi_1}^*, ~\widetilde{U}^\dagger =  \frac{1}{\#H_2}\sum_{b \in G}A(b^{-1}) \widetilde{\pi_1}_*T_{b^{-1}}\widetilde{\pi_2}^*,
$$
the pushforward Liouville volume density on $S^*M_2$ is 
\begin{align*}
\widetilde{U}_* d\mu_{L_1} &= \widetilde{U}^{\dagger *} d\mu_{L_1}   \\
&=\frac{1}{\#H_2}\sum_{h_2\in H_2}\sum_{b\in G} A(b^{-1}) (d\widetilde{\pi_1}_*)^*(dT_{b^{-1}})^* (dT_{h_2})^*d\mu_L\\
&=\frac{1}{\#H_2}\sum_{h_2\in H_2}\sum_{b\in G} A(b^{-1}) (d\widetilde{\pi_1}_*)^*(dT_{b^{-1}h_2})^*d\mu_L\\
&= \frac{1}{\#H_2}\sum_{h_2\in H_2} (dT_{h_2})^*d\mu_L=  d\mu_{L_2}. 
\end{align*}
Likewise, 
\begin{equation}\label{pushforwardmu_2}
    \widetilde{U}^\dagger_{~*}d\mu_{L_2} = d\mu_{L_1}.
\end{equation}

Fix $t\in \R^+$, let $u,v \in   \mathcal{D}(S^*M_1)$, and $k^t_1$, $k^t_2$ be respectively the distributional Schwartz kernel of the transfer operators $V^t_{g_1}$, $V^t_{g_2}$. By Theorem \eqref{Vtheorem},
\begin{align*}
    \left( k_1^t , u\otimes v \right) &= \left(V_{g_1}^t v ,u \right)_{L^2(S^*M_1,d\mu_{L_1})}\\
    &= \left(\widetilde{U}^*V_{g_2}^t\widetilde{U}v , u \right)_{L^2(S^*M_1,d\mu_{L_1})}\\
    &= \left( V_{g_2}^t\widetilde{U}v , \widetilde{U}u \right)_{L^2(S^*M_2,d\mu_{L_2})}
=\left( k_2^t, \widetilde{U}u \otimes \widetilde{U}v\right) . 
    \end{align*}
Let $\Delta_i: S^*M_i \rightarrow S^*M_i \times S^*M_i,~ i =1,2$ be the diagonal maps. It follows that
\begin{align*}
    \left(\widetilde{U}^{\dagger *}\Delta_1^*k^t_1, \widetilde{U}\phi\right) =\left(\Delta_1^*k^t_1, \phi\right) &= \left(k^t_1, \Delta_{1*}\phi\right)\\
    &= \left(k^t_2, \widetilde{U}\Delta_{1*}\phi\right)\\
    &= \left(k^t_2, \Delta_{2*}\widetilde{U}\phi\right) = \left(\Delta_2^*k^t_2, \widetilde{U}\phi\right)
\end{align*}
for any $\phi \in \mathcal{D}(S^*M_1)$. By unitarity, 
\begin{equation}\label{pushforwarddeltak_1}\widetilde{U}^{\dagger *}\Delta_1^*k^t_1  =  \Delta_2^*k^t_2 \end{equation}
as a $\delta$-distribution on $S^*M_2$.

Denote $\Pi_i: S^*M_i \times \R^+ \rightarrow S^*M_i,~i   =1,2$ to be the natural projections. For any $\varphi \in \mathcal{D}(\R^+)$, we combine \eqref{pushforwardmu_2} and \eqref{pushforwarddeltak_1} to obtain
\begin{align*}
\left(\Pi_{1*}\Delta_1^*k_1^\bullet, \varphi\right) = \left((\widetilde{U}^{\dagger}_*\Pi_2)_*\Delta_1^*k_1^\bullet, \varphi \right) =  \left(\Pi_{2*}\widetilde{U}^{\dagger*}\Delta_1^*k_1^\bullet , \varphi\right) =\left(\Pi_{2*}\Delta_2^*k_2^\bullet  , \varphi\right)
\end{align*}
and hence $\textup{Tr}^\flat V^\bullet_{g_1} = \textup{Tr}^\flat V^\bullet_{g_2}$ as $\delta$-distributions on $\R^+$.

\end{proof}


\section{\centering \sc{ Proof of theorem 1.2}}

To establish Theorem \ref{Vtheorem}, our goal is to show that $\widetilde{U}_A$ satisfies  $\widetilde{U}_AV_{g_1}^t=V_{g_2}^t\widetilde{U}_A$. 
\medskip
\begin{proof}
Let $\Xi_{H^i}$, $i=1,2$, respectively, $\Xi_H$ be the contact Hamiltonian fields on each $S^*M_i$ and $S_g^*M$. Observe that $V^t_{g_i}$ can be obtained via conjugating $V_g^t$ with the pull-back $\widetilde{\pi}^*_i$. Indeed, given any point $\widetilde{\zeta}$ on $S^*M$, by uniqueness of the local geodesic flow, 
$f(\text{exp}(t\Xi_H(\widetilde{\pi}_i(\widetilde{\zeta})))
    = f(\text{exp}(t\Xi_{H^i}(\widetilde{\pi}_i(\widetilde{\zeta})))$. Hence,  
\begin{equation*}
    V_g^t(\widetilde\pi_i^*f)(\widetilde\zeta) 
    = (\widetilde \pi_i^*f)\circ G^t_g(\widetilde\zeta) =
    f(\text{exp}(t\Xi_{H^i}(\widetilde\pi_i(\widetilde\zeta)))
    = \widetilde\pi_i^*V_{g_i}^tf(\widetilde\zeta)     
\end{equation*}
for any $f \in L^2(S^*_gM)$.

\smallskip

Similarly, the conjugation applies with the push-forward operator $\pi_{i_*}$. Indeed, let $\zeta'$ be a point on $S_{g_i}^*M$. For $f$ an $L^2$-function on $S_g^*M$, we have

\begin{equation*}
    \widetilde{\pi}_{i*}V_g^tf(\zeta')
    = \sum_{\eta\in\widetilde\pi_i^{-1}(\zeta')}f\circ G_g^t(\eta) 
    =\sum_{h_i\in H_i} f\circ G_g^t(\widetilde T_{h_i}\eta),
\end{equation*}
and the equality is independent of the choice of $\eta$ along the fibre of $\zeta'$ by transitivity. Furthermore, 
for any preimage $\zeta$ of $\zeta'$, the orbit of $\zeta$ is the preimage of the orbit of $\zeta'$, which gives  
\begin{align*}
V_{g_i}^t\widetilde{\pi}_{i*}f(\zeta') 
&=\sum_{h_i\in H_i}f \circ G_g^t (\widetilde{T}_{h_i}\eta') = \widetilde{\pi}_{i*}V_g^tf(\zeta').
\end{align*}
Combining the equalities and the fact that $\widetilde{T}_a$ commutes with the Koopman operator as the deck group transformations $G$ acts on the cotangent bundle as isometries to attain 
\begin{align*}
    \widetilde{U}_AV_{g_1}^t 
&= \frac{1}{\# H_1}\sum_{a \in G }A(a)\widetilde{\pi}_{2*}\widetilde{T}_aV_g^t\widetilde{\pi}_1^*\\
&=\frac{1}{\# H_1}\sum_{a \in G }A(a)\widetilde{\pi}_{2*}V_g^t\widetilde{T}_a\widetilde{\pi}_1^*\\
&=\frac{1}{\#H_1}\sum_{a \in G }A(a)V_{g_2}^t\widetilde{\pi}_{2*}\widetilde{T}_a\widetilde{\pi}_1^*
=V_{g_2}^t\widetilde{U}_A. 
\end{align*}

\end{proof}

\section{\centering \sc{Appendix}}\label{appendix}

\subsection*{\normalsize \textbf{$\delta$-distributions}.}\label{deltadensity} The goal of this section is to describe the Schwartz kernel of the Koopman operator and its pullback $\Delta^*k^\bullet$ as distributions. We refer to \cite{G77}, \cite{GuSc90} for further background. Let us first recall the general settings of the so-called $\delta$-distributions. 

Let $X$ be a manifold. Denote $|T|X$ to be the density bundle of $X$ whose fibre elements at a point $x\in X$ are the volume densities on $X$. Let $L$ be a line bundle of $X$ and $L^*$ be its dual. A \textit{generalized section} of the line bundle $L \rightarrow X$ is a continuous linear functional on the space of $C^\infty$-sections of $L^*\otimes |T| X$. These are volume forms with compact support in $X$ taking values in the co-line bundle $L^*$. In particular, if $L= \R$, generalized sections of $L$ are \textit{generalized functions} on $X$. Namely, these are elements in $\textup{Hom}(C_c^\infty(|T|X), \R)$. If $L = |T|X $, generalized sections are now \textit{generalized densities} on $X$ or elements of  $\textup{Hom}(C_c^\infty(X), \R)$ as $L^* \otimes L$ is canonically trivial. 

Now let $Z$ be a closed submanifold of $X$, the conormal bundle $N^*Z$ is given by 
\[
N^*Z = \{(\zeta, \alpha)\in Z \times T^*S^*_gM \big| \alpha \in T^* : \alpha(v) = 0, \forall v \in T_\zeta Z \}.
\]
Given a vector bundle $E\rightarrow X$, a \textit{$\delta$-section} along $Z$ is a smooth generalized section $u$ of the tensor bundle
\[
E\big|_{Z} \otimes |N^*Z|\]
such that 
\begin{equation}\label{pairing}
    \langle u, \psi \rangle = \int_{Z}u \psi\big|_{Z}
\end{equation}
for any section $\psi$  of $E^* \otimes |T|X$. Let $E = |T|X$.
We have the short exact sequence 
\begin{equation}\label{shortexact}
    0 \rightarrow T_\zeta Z \rightarrow T_\zeta X \rightarrow T_\zeta X/T_\zeta Z \rightarrow 0.  
\end{equation}
Since $|T|\big|_{Z}X \otimes |N^*Z| \cong |T|Z$, the pairing in \eqref{pairing} implies that $u$ is a generalized density on $X$ with support in $Z$ taking values in the space of densities on $Z$.  

\medskip

Next we recall pull-backs of $\delta$-sections under the condition of transversality. Let $Y$ be another manifold and $F \rightarrow Y$ a vector bundle. Consider a smooth map $h : X \rightarrow Y$. Let $W$ be a submanifold of $Y$ defined as a zero-set of a global chart on $Y$. If $h \pitchfork W$, it follows that $h^{-1}W$ is a submanifold of $X$ and that $dh : T_xX \rightarrow T_yY $ maps $ T_xh^{-1}(W)$ onto $T_yW$, which induces the mapping
\begin{equation}\label{isomconormal}
   N^*_x h^{-1}W \rightarrow N^*_y W ;~ \beta \mapsto dh^*\beta.
\end{equation}
If we let $\sigma$ to be a $\delta$-section along $W$, that is,  a generalized section of $F\big|_W \otimes |N^*W|$, and $r(y): F_y \rightarrow E_x$ be a fiber map, then 
 $r \otimes dh^*|\wedge| : F_y \otimes |N^*W| \rightarrow E_x \otimes |N_x^*h^{-1}W|$  maps $\sigma$ to its pullback $h^*\sigma$, which is a $\delta$-section along $h^{-1}W$.

\medskip
 
Let $p_1, p_2$ be the projections of $X\times Y$ to $X$ and respectively, to $Y$. 
Let $k = k(x,y)dy$ be a generalized density  of $p_2^*|T|Y$, that is a linear functional on $C_c^\infty Y$. We obtain from $k$ a linear operator 
\[
K a = p_{1*} p_2^* \langle a, k \rangle = p_{1*}  \langle a, p_2^* k \rangle \in C^\infty X ; ~~~ a\in C_c^\infty Y.
\]

\begin{theorem*}[Schwartz kernel theorem]
Every continuous linear operator is of this form.
\end{theorem*}

In particular, if $f : X \rightarrow Y$ is a smooth map, the composition operator $f^*$ is  
\[f^* a(x) = a (f(x)) = \int_{Y}\delta(y - f(x))a(y) dy.  \]
$k= \delta (y - f(x))dy$ is then referred to as a  \textit{$\delta$-distribution} supported on the graph of $f$.  

\medskip

\textbf{Remark. } This description of $k$ is consistent with the definition of $k$ as a $\delta$-section of 
\[ p_2^*|T|\big|_{\textup{Graph}(f)} Y \otimes |N^*\textup{Graph}(f)|. \]
A 1-form $\beta \in T_y^*$ pullbacks under $p_2$, which satisfies the condition
\[\beta_y + (-df_x)^*\beta_y= 0\]
induced by the defining equation $y = f(x)$ for the submanifold $\textup{Graph}(f) \subset X \times Y$. Thus, if $(\alpha_x, \beta_y) \in p_2^*T^*\big|_{ \textup{Graph}(f)} Y$, it is necessary that $\alpha_x = (-df_x)^* \beta_y$. This is precisely the condition for $(\alpha, \beta)$ to be an element in $N^*\textup{Graph}(f)$, which is

\begin{equation*}
    (\alpha ~ \beta) \cdot \begin{pmatrix} V_x \\ df_x (V) \end{pmatrix}  = \alpha(V) + \beta df_x (V) = 0.
\end{equation*}
Thereby, 
\[ p_2^*|T|\big|_{\textup{Graph}(f)}Y \otimes |N^*\textup{Graph}(f)| \cong 1,\] which implies that $k$ is the canonically trivial distribution on $X\times Y$ supported on $\textup{Graph}(f)$ taking values in the line of densities on $\textup{Graph}(f)$.

\medskip

Now let $X = Y = S_g^*M$ and $p_2 = \rho (\eta, \xi ,t) \mapsto \xi$.

\begin{equation*}
\begin{tikzcd}[row sep=huge]
\rho^*T^*S^*_gM  \arrow[d,""] & T^*S^*_gM \arrow[d,swap,""]
\\
S^*_gM \times S^*_gM \times \R^+ \arrow[r,"\rho"] & S^*_gM 
\end{tikzcd}
\end{equation*}
From the previous paragraphs, the Schwartz kernel of $V^t_g$ is the $\delta$-distribution - that is the $\delta$-section of the tensor bundle 
\[\rho^*|T|\big|_{\textup{Graph}(G^t_g)}S^*_gM \otimes \lvert N^* \textup{Graph}(G^t_g)\rvert\]
given by 
\[k^t = k^t(\xi, \eta)d\mu_L(\eta)d\mu_L(\xi) = \delta(\eta - G^t_g \xi) d\mu_L(\eta)d\mu_L(\xi).\]
Under transversality, $\Delta^* k^\bullet$ is a $\delta$-section of the tensor bundle 
\[\Delta^*\rho^*|T|\big|_{\textup{Fix}(G^\bullet_g)}S^*_gM \otimes \lvert N^* \textup{Fix}(G^\bullet_g)\rvert.\]
Indeed, we first note that $\Delta^{-1}\textup{Graph}(G^\bullet_g) =\textup{Fix}(G^\bullet_g) = \{(\zeta, \tau) : \zeta \in \textup{Fix}(G_g^\tau), \tau \in \textup{Lsp}(S_g^*M)\}$.
Similar to the the discussion leading up to \eqref{isomconormal}, letting $h = \Delta$ yields the canonical isomorphism 
\[N^*\gamma \cong \Delta^* \rho^* T^* \big|_{\gamma}S_g^*M \]
if $\gamma$ is a non-degenerate $\tau$-periodic orbit of $G^t_g$ whereby $\Delta\big|_\gamma \pitchfork \textup{Graph}G^\tau_g$ . More generally,

\[N^*Z^{(\tau)}_j \cong \Delta^* \rho^* T^* \big|_{Z^{(\tau)}_j} S_g^*M\]
if $Z^{(\tau)}_j$ is a clean $k$-dimensional component of $\textup{Fix} (G^{\tau}_g)$. To see this, notice for any 1-form $(\alpha, \beta)$ in $\rho ^*T^* S^*_gM$, we have 
\[\Delta_*(\alpha,\beta) = (d \Delta)^* (\alpha, \beta) = \alpha + \beta.\]
Hence, since $dG^\tau_g$ fixes the tangent space along $\gamma$, 
\[
    \Delta^* ((-dG_g^{\tau})^*\beta, \beta) = (I - dG^\tau_g)^* \beta 
\]
is a 1-form in $(\Delta^*\rho^* T^*\big|_{Z^{(\tau)}_j} S^*_gM)_\zeta$, which vanishes on $T_\zeta Z_j^{(\tau)}$ under cleanness, hence is in the conormal space $N_{\zeta}^*Z_j^{(\tau)}$.  Now let $\alpha_\zeta \in N_{\zeta}^*Z_j^{(\tau)}$, it is obvious that $\alpha\big|_{T_\zeta Z_j^{(\tau)}} = (I - dG^\tau_g)^*\alpha_\zeta$. Suppose $J \in \ker(\alpha_{\zeta}) - T_{\zeta}Z_j^{(\tau)}$, since $(I - dG^\tau_g)^\#$ is invertible on $T_\zeta S^*M/ T_{\zeta}Z_j^{(\tau)}$, there exists $S$ in the cokernel of $I - dG^\tau_g$ so that 
\[\alpha_\zeta(J) = (I - dG^\tau_g)^*\alpha_\zeta(S).\]
This implies $\alpha = (I - dG^\tau_g)^*\beta$ for some $\beta \in (\Delta^*\rho^*T^*)_\zeta$. We thus obtain an isomorphism of $(\Delta^*\rho^*T^*|_{Z^{(\tau)}_j})_{\zeta}$ with $N_{\zeta}^*Z_j^{(\tau)}$. Consequently,
\[|N^*Z_j^{(\tau)}| \otimes \Delta^*\rho^*|T|\big|_{Z^{(\tau)}_j} = 1.\]
Moreover, since $d\Delta_\zeta +  (I_\zeta,(dG^\tau_g)_\zeta^\#)$ is full-rank, this transversal condition implies that  $k^\tau$ pulls back to the canonical $\delta$-distribution $\Delta^*k^\tau$ on $S_g^*M$, which is supported on $Z_j^{(\tau)}$ taking values to densities on $Z_j^{(\tau)}$ (cf. [\cite{G77}, Theorem 6]). In both cases of Lefschetz and clean flows, 
\[\Delta ^*k^\bullet\big|_{(\zeta, \tau)} = \delta(\zeta - G_g^\tau\zeta) d\mu_L(\zeta) \cdot  \delta(t-\tau)dt.\]

\medskip

\begin{noindent}
Northwestern University
\end{noindent}

\begin{noindent}
Department of Mathematics
\end{noindent}

\begin{noindent}
2033 Sheridan Road, Evanston, Illinois 60208 USA
\end{noindent}

\medskip

\begin{noindent}
\textit{E-mail address}: \texttt{hylam2023@u.northwestern.edu}
\end{noindent}

 \end{document}